\documentclass[12pt,a4paper]{amsart}

\usepackage[margin=3cm,heightrounded]{geometry}

\usepackage{latexsym}
\usepackage{amsfonts}
\usepackage{amsmath,amsthm,amssymb,mathtools}
\usepackage{graphicx}

\newcommand{\cal}{\mathcal}

\newcommand{\Prob}{{\mathbb P}}

\newcommand{\Exp}{{\mathbb E}}

\newcommand{\dd}{{\mathrm{d}}}
\newcommand{\e}{{\mathrm{e}}}

\newcommand{\oh}{{\mathrm{o}}}
\newcommand{\Oh}{{\mathrm{O}}}

\newcommand{\halmoss}{{\mbox{\, \vspace{3mm}}} \hfill
\mbox{$\Diamond$}}
\newcommand{\RR}{\mathbb {R}}

\newcommand{\PP}{\mathbb {P}}

\newcommand{\EE}{\mathbb {E}}

\newtheorem{Th}{Theorem}[section]
\newtheorem{Lemma}{Lemma}[section]

\newtheorem{Prop}{Proposition}[section]
\newtheorem{Cor}{Corollary}[section]

\newtheorem{Rem}{Remark}[section]

\newcommand{\logsim}{\approx_{\rm log}}

\newcommand{\restart}{RESTART}%{{\footnotesize restart\;}}

\newcommand{\indi}[1]{{\mathbf 1}\bigl\{{#1}\bigr\}}

\newcommand\RESTART{\textsc{restart}}
\newcommand{\N}{{\cal N}}

\newcommand{\D}{{\mathrm d}}

\begin{document}
\title{Time inhomogeneity in longest gap and longest run problems}

%\runauthor{H.\ Albrecher and J.\ Ivanovs}

\author[S. Asmussen]{S\o ren Asmussen}
\email{asmus@math.au.dk}
\address{Aarhus University}
\author[J. Ivanovs]{Jevgenijs Ivanovs}
\email{jevgenijs.ivanovs@unil.ch}
\address{University of Lausanne}
\author[A. R. Nielsen]{Anders R\o nn Nielsen}
\email{arnielsen@math.ku.dk}
\address{University of Copenhagen}

%\thankstext{m1}{}

%\affiliation{Department of Actuarial Science, University of Lausanne, Lausanne CH-1015, Switzerland\thanksmark{m1}, and\\ Swiss Finance Institute, University of Lausanne, Lausanne CH-1015, Switzerland\thanksmark{m2}}

%\affiliation{University of Lausanne\thanksmark{m1} and Swiss Finance Institute\thanksmark{m2}}

\begin{abstract}
Consider an inhomogeneous Poisson process and let $D$ be the first of its epochs which is followed by a gap of size $\ell>0$.
We establish a criterion for $D<\infty$ a.s., as well as for $D$ being long-tailed and short-tailed, and obtain logarithmic tail asymptotics in various cases.
These results are translated into the discrete time framework of independent non-stationary Bernoulli trials
where the analogue of $D$ is the waiting time for the first run of ones of length $\ell$.
A main motivation comes from computer reliability, where
$D+\ell$ represents the actual execution time of a program or transfer of a file of size $\ell$ in presence of failures (epochs of the process) which necessitate restart.
\end{abstract}

%\subjclass[2010]{Primary 60G51; Secondary 91B30}


\thanks{Financial support by the Swiss National Science Foundation Project 200020 143889 is gratefully acknowledged.}

\maketitle

%%%%%%%%%%%%%%%%%%%%%%%%%%%%%%%%%%%%%%%%%%%%%%%%%%%%%%%%%%%%%%%%%%%%%

\section{Introduction}\label{S:Intr}
This paper is concerned with the study of the time
\[D\ =\ \min\{T_n:\,T_{n+1}-T_n\geq \ell\}\]
of occurrence of the first gap of length $\ell$ in an inhomogeneous Poisson process $\N$ on $(0,\infty)$ with epochs
 $0<T_1<T_2<\cdots$ (where we use the convention $T_0=0$). In particular, we study the logarithmic asymptotics of the tail $\Prob(D>t)$ as $t\to\infty$ subject to a variety of forms of the rate function $\mu(t)$ of~$\N$. 
 
 The tail probability $\Prob(D>t)$ can alternatively be written as $\Prob\bigl(L(t+\ell)< \ell\bigr)$ where
 \[L(t)\ =\ \sup\{T_{n+1}\wedge t-T_n:\,T_n< t\}\]
 is the longest gap between epochs before $t$. In this formulation, the time-homogeneous problem 
 where $\mu(t)\equiv\mu$ has a classical
 discrete time parallel as the study of the longest success run $L_n$ of 
 $n$ i.i.d.\ Bernoulli trials $\xi_1,\ldots,\xi_n$, with $\Prob(\xi_k=0)$ taking the role of $\mu$. This is an old and a well-studied problem with applications to insurance, finance, traffic and reliability, see~\cite{balakrishnan2011runs,erdos,vonMises}, nevertheless there is hardly any literature on the inhomogeneous case.
 In the body of the paper, we consider the continuous time Poisson framework but outline the
 translation to  the inhomogeneous Bernoulli case (where zeroes take the role of epochs)
 %$\Prob(\xi_k=0)$ depends on $k$ 
 in Section~\ref{sec:discrete}.

 The asymptotics of the tail $\Prob(D>t)$ is fairly easy to obtain  in the homogeneous  case
 %, i.e.\ $\mu(t)=\mu$, 
 where a renewal argument easily gives $\PP(D>t)\sim c\e^{-\gamma t}$ as $t\rightarrow\infty$, with $\gamma$ being a root of a certain equation, see Proposition~\ref{prop:homogen} below.
In contrast, the behaviour is more diverse in the inhomogeneous case, and it may even happen that
$\Prob(D<\infty)<1$. 
In Section~\ref{S: Dfin}, we show that the critical
rate of increase of $\mu(t)$ for this phenomenon is $\ell \log t$.
Thus the rate
of increase to $\infty$ of $\mu(t)$ can only be allowed to be very modest for $D$ to be finite a.s. In addition, in Section~\ref{sec:long_short} we show that if $\mu(t)\rightarrow\infty$ then $D$ has a long-tailed distribution, i.e.\ $\PP(D>t+s)/\PP(D>t)\rightarrow 1$ as $t\rightarrow\infty$, whereas $\PP(D>t+s)/\PP(D>t)\rightarrow 0$ when $\mu(t)\rightarrow 0$.

Our asymptotic study is presented in Section~\ref{sec:asymptotics}, where we separately discuss the following three cases: 
(i) $\mu(t)\rightarrow\mu$, (ii) $\mu(t)\rightarrow 0$ and (iii) $\mu(t)\rightarrow\infty$.
Note that (ii) includes the case  where $\int_0^\infty \mu(t)\,\dd t<\infty$ so that there is a last epoch
$T^*<\infty$ of $\cal N$. It could then happen that $D=T^*$, but our results (based on a bounding argument) show that typically the tail of $D$ is lighter than that of $T^*$.
Particular examples studied are $\mu(t)=a\log^{-b} t$, $\mu(t)=at^{-b}$ and $\mu(t)=a\e^{-b t}$. The long-tailed case (iii) is analyzed using a delay differential equation derived in Section~\ref{sec:exact}; a particular example 
is $\mu(t)=b\log t$ for $b\in (0,1/\ell)$. We also identify a critical rate separating the cases when $\EE D^p$ is finite or infinite.
Section~\ref{sec:Restart} deals with what provided our initial motivation, the study of the tail of the total execution
time $X$ of a task like program or file transmission in a fault-tolerant computing environment
working under the \RESTART{} protocol, where
the task needs to be completely restarted after failure. Here $\ell$ takes the role of the ideal task time,
failures occur at the epochs of $\cal N$ and so $X=\ell+D$. Earlier studies of similar problems are in
Asmussen \emph{et al.}~\cite{A5,Buda,Olebook} and  Jelenkovi\'c  \emph{et al.}~\cite{JT2,JT3}. The novelty here
is the time-inhomogeneity.
We also discuss a related \RESTART{} problem with homogeneous failures, but time-varying service rate $r(t)$.
A main idea is to use
a simple time change
to transform $\N$  to a homogeneous Poisson$(\mu)$ process.

Finally, Section~\ref{sec:discrete}
gives the corresponding results for the discrete time inhomogeneous Bernoulli  case.
Intuitively, this is connected to the Poisson framework by $p_i=\Prob(\xi_k=1)=\e^{-\mu(i)}$, which is roughly the probability of no failures in $(i-1,i]$, and with one exception,
the analysis is indeed a straightforward translation. Classical references such as ~\cite{balakrishnan2011runs,erdos,vonMises}
only treat time-homogeneity. Time-inhomogeneity only seems to have been studied in the 
framework of Markovian regime switching which is somewhat different from the models of this paper
by being asymptotically stationary rather than exhibiting a trend. Some references are
\cite{antzoulakos1999waiting,fu1994distribution,{wallclock},Olebook} (\cite{wallclock} also contains some early
and in part unprecise version of a few of the results of this paper).

\subsubsection*{Preliminaries}
We represent the Poisson point process $\N$ on $(0,\infty)$ as a random subset of  $(0,\infty)$. The intensity measure is denoted $M(\dd t)$ and taken absolutely continuous with respect to Lebesgue measure on $(0,\infty)$, i.e.\ $M(\D t)=\mu(t)\D t$ for some rate function $\mu(t)$; we also write $M(a,b)=\int_a^b\mu(t) \dd t$. 
If we write $\N(a,b)$ for the number of points in $(a,b)$, we thus have
\begin{equation}\label{23.9a}
\Prob\bigl(\N(a,b)=0\bigr)=\e^{-M(a,b)}\,,\quad \Prob\bigl(\N(a,b)\ge 1\bigr)=1-\e^{-M(a,b)}\le M(a,b)
%=\Exp\N(a,b)
\,.
\end{equation}
Moreover, it is assumed that $M(0,t)<\infty$ for any $t$, so that $\N(a,b)<\infty$ a.s.\ when $b<\infty$,
%with probability one any finite interval contains finitely many points of $\N$, 
and that $\mu(t)>0$ for (Lebesgue) almost all $t\geq 0$. The latter assumption guarantees that 
$\PP(D>t)>0$ for any $t>0$, and can be replaced by a weaker one.

\begin{Rem}\label{rem:l=1}\rm
Given the intensity measure $M(\dd t)$ and $\ell$, we may scale down time by $\ell$ and consider the new point process
${\cal N}'={\cal N}/\ell$. That is, we may take $\ell'=1$ and $M'(0,t)=M(0,t\ell)$, yielding $\mu'(t)=\ell\mu(t\ell)$ and
(in obvious notation)
\begin{equation}\label{27.8a}
\Prob\bigl(D>t\,\big|\,\ell,\mu(\cdot)\bigr)\ =\ \Prob(D>t) \ =\ 
\Prob(D'>t/\ell)\ =\ \Prob\bigl(D>t/\ell\,\big|\,1,\ell\mu(\cdot\ell)\bigr)
\end{equation}
Nevertheless, we formulate all our results for a general $\ell$, but sometimes switch to $\ell=1$ in the proofs.\halmoss
\end{Rem}

Finally, the relation $f(x)\sim g(x)$ means $f(x)/g(x)\to 1$ and  $f(x)\logsim g(x)$  logarithmic asymptotics
as in large deviations theory, i.e.\ $\log f(x)/\log g(x)\to 1$.

\section{First calculations}\label{sec:exact}
We start by recalling the famous Slivnyak's formula of Palm theory, see e.g.~\cite{stoch_geom}, which is the basic tool in most of our calculations. For any non-negative (measurable) function $h$ it states that
\[\EE\sum_{t\in \N}h(t,\N\backslash\{t\})=\int_0^\infty \EE h(t,\N)\mu(t)\,\dd t.\]
In this setting the indicator that there are no gaps in $(0,t)$ will often be useful:
\begin{equation}\label{eq:indicator}%g(t,\N)=
\indi{L(t)<\ell}=\indi{T_{n+1}\wedge t-T_n<\ell,\forall T_n\in[0,t)}=\indi{D>t-\ell},\end{equation} which depends only on the points of $\N$ in $[0,t)$ and hence  is independent of $\N\cap [t,\infty)$.

We first present a delay differential equation for the tail probabilities.
It will be used for a crucial estimate in Section~\ref{SS:Long_tail} and is also potentially useful 
for computations of exact values of the $\Prob(D>t)$.
\begin{Prop}\label{prop:calc}
It holds for $t\geq 0$ that
\[\PP(D\in(t,\infty))=\int_t^\infty\e^{-M(s,s+\ell)}\PP(D>s-\ell)\mu(s)\,\dd s.\]
\end{Prop}
\begin{proof}
Consider the event $D\in(t,\infty)$, which means that there is a point $s>t$ followed by a gap, and additionally
there are no gaps in $(0,s)$. There can be only one such location~$s$ and hence only one such point of $\N$ a.s.
Hence by Slivnyak's %and void probability 
formula we readily obtain
\begin{align*}\PP(D\in(t,\infty))&=\EE\sum_{s\in \N,s>t}\indi{\N\cap (s,s+\ell)=\emptyset}\indi{\text{no gaps in }(0,s)}\\
&=\int_t^\infty\PP(\N\cap(s,s+\ell)=\emptyset)\PP(D>s-\ell)\mu(s)\,\dd s\\&
=\int_t^\infty\e^{-M(s,s+\ell)}\PP(D>s-\ell)\mu(s)\,\dd s\,.
\end{align*}
Here we used that the above indicators are independent and stay unchanged when we remove $s$ from $\N$, since $(0,s)$ and $(s,s+\ell)$ are disjoint and do not contain $s$.
\end{proof}

Differentiating the result of Proposition~\ref{prop:calc} at $t$, we obtain the  delay differential equation
\begin{equation}\label{eq:diff_eq}
\Prob(D>t)'=-\e^{-M(t,t+\ell)}\PP(D>t-\ell)\mu(t),\quad t\geq 0.
\end{equation}
This may be solved in the intervals $\bigl[k\ell,(k+1)\ell\bigl)$ by using recursion in  $k$ and the initial
condition \begin{align}\label{20.9a}
&\Prob(D>t)=1-\int_0^t\e^{-M(s,s+\ell)}\mu(s)\dd s-\e^{-M(0,\ell)}, &t\in[0,\ell),
\end{align}
%\eqref{20.9a} 
which follows from $\PP(D=0)=\e^{-M(0,\ell)}$ and the consequence
\[\PP(D\in(0,t])\ =\ \int_0^t\e^{-M(s,s+\ell)}M(\dd s)\,,\quad t\in[0,\ell]\]
of Proposition~\ref{prop:calc}.
Letting $f(t)=-\log\Prob(D>t)$ we also have
\begin{equation}\label{eq:difflog}f'(t)=\e^{-M(t,t+\ell)}\mu(t)\e^{f(t)-f(t-\ell)},\end{equation}
which may be more suitable for numerical computation.

%OLD
%Proposition~\ref{prop:calc} implies for $t\in[0,\ell]$ that
%$\PP(D\in(0,t])=\int_0^t\e^{-M(s,s+\ell)}M(\dd s)$. Moreover, $\PP(D=0)=\e^{-M(0,\ell)}$ and so
%\begin{align}
%&\Prob(D>t)=1-\int_0^t\e^{-M(s,s+\ell)}\mu(s)\dd s-\e^{-M(0,\ell)}, &t\in[0,\ell].
%\end{align}
%Differentiate the result of Proposition~\ref{prop:calc} at $t$ to obtain the following delayed differential equation:
%\begin{equation}\label{eq:diff_eq}
%\Prob(D>t)'=-\e^{-M(t,t+\ell)}\PP(D>t-\ell)\mu(t),\quad t\geq 0.
%\end{equation}
%Letting $f(t)=-\log\Prob(D>t)$ we also have
%\begin{equation}\label{eq:difflog}f'(t)=\e^{-M(t,t+\ell)}\mu(t)\e^{f(t)-f(t-\ell)},\end{equation}
%which may be more suitable for numerical computation.
%

\section{Finiteness of $D$}\label{S: Dfin}
First, we present an integral test.
\begin{Th}\label{prop_int_test} If $M(0,\infty)<\infty$ then $D<\infty$ a.s. Otherwise, let
 \[I=\int_0^\infty \e^{-M(t,t+\ell)}\mu(t)\dd t\,.\] Then $D<\infty$ a.s.\ if $I=\infty$ and   $\PP(D=\infty)>0$ if $I<\infty$.
\end{Th}
\begin{proof} 
It follows from Proposition~\ref{prop:calc} for $t=0$ that
\[1-\PP(D=\infty)\geq \PP(D\in(0,\infty))\geq \PP(D=\infty)I.\]
Hence $\PP(D=\infty)\leq 1/(1+I)$ and so $D<\infty$ a.s. if $I=\infty$.
Next, if $M(0,\infty)<\infty$ then $T^*<\infty$ and hence $D<\infty$ a.s.\ as noted in the Introduction. 
It thus remains to consider the case $M(0,\infty)=\infty$ and $I<\infty$.

%
%According to Remark~\ref{rem:l=1}
%\[I'=\int_0^\infty \e^{-M'(t,t+1)}\mu'(t)\,\dd t=\int_0^\infty \e^{-M(t\ell,(t+1)\ell)}\ell\mu(t\ell)\dd t=I\] and so it is enough to show the claim for $\ell=1$.
%Let $E_t$ be the event that $\N$ has a point in $[t,t+1)$ followed by a gap; there can be only one such point. By Slivnyak's
%% and void probability
%formula,
%\begin{align*}
%\PP(E_t)&=\EE\sum_{s\in\N}\indi{s\in[t,t+1)}\indi{\N\cap(s,s+1)=\emptyset}\\
%&=\int_t^{t+1}\PP(\N\cap(s,s+1)=\emptyset)\mu(s)\,\dd s=\int_t^{t+1}\e^{-M(s,s+1)}\mu(s)\,\dd s.\end{align*}
%Hence
%\[I=\sum_{n=0}^\infty\PP(E_{n})=\sum_{n=0}^\infty \PP(E_{2n})+\sum_{n=0}^\infty \PP(E_{2n+1}).\]
%If $I=\infty$ then at least one sum on the right is infinite; suppose the first one. But the events $E_{2n}$ are independent, and so by Borel-Cantelli lemma infinitely many of them occur with probability~1. In particular, $D<\infty$ a.s.
%
%Assume conversely $I<\infty$.
Observe that the probability of a gap in $(T,\infty)$ is bounded above 
by the expected number $\Exp N(T)$ of gaps in $(T,\infty)$, where
\begin{equation}\label{26.8a}
\Exp N(T)\ =\ \exp(-M(T,T+\ell))\,+\,\int_T^\infty \exp(-M(t,t+\ell))\mu(t)\D t\,.
\end{equation}
We choose $T$ so large that the last term is smaller than $1/4$, which is possible according to $I<\infty$. Now if the first term is smaller than $1/4$ then the probability of no gaps in $(T,\infty)$ is at least $1/2$.
Hence we can define $A=\{t\ :\,\exp(-M(t,t+\ell))\le 1/4\}$ and note that $A\cap(t,\infty)$ has positive Lebesgue measure for any $t$ since otherwise we have
a contradiction with the assumptions $I<\infty $ and $M(0,\infty)=\infty$. 
Thus we can choose $T'>T$ such that $S=(T',T'+\ell/2)\cap A$ has  positive Lebesgue measure. Finally, assuming that $S$ has a point and conditioning on the first such point we readily obtain the inequality
\begin{align*}\PP(D=\infty)&\geq \int_S \PP(\text{no gaps in }(0,t),\mathcal N\cap (T',t)=\emptyset)\PP(\text{no gaps in }(t,\infty))\mu(t)\D t\\
&\geq \frac{1}{2}\int_S \PP(\text{no gaps in }(0,t),\mathcal N\cap (T',t)=\emptyset)\mu(t)\D t\,,\end{align*}
according to the choice of~$S$.
But the probability under the last integral sign is positive for any $t\in S$, which readily implies $\PP(D=\infty)>0$.

%To complete the proof, consider the stopping time $\tau\ =\ \inf\{t>T:\,t\in{\cal N}\cap A\}$ which is finite a.s.\ and hence $\PP(\text{no gaps in }(0,\tau))=\delta>0$.
%Applying the strong
%Markov property at $\tau$ (or using Slivnyak's formula), we finally obtain $\Prob(D=\infty)\geq \delta/2>0$.
\end{proof}

Some comments with regard to Theorem~\ref{prop_int_test} may be useful. Note that $\PP(D=\infty)>0$ if and only if $I<\infty$ and $M(0,\infty)=\infty$, because $M(0,\infty)<\infty$ implies $I<\infty$. Intuitively, there are three regimes: low (tail) rate corresponding to $M(0,\infty)<\infty$, moderate rate corresponding to $I=\infty$, and high rate otherwise; and it is the last regime which leads to an infinite $D$ with positive probability. The difficulty in applying Theorem~\ref{prop_int_test} directly is that the integrand in~$I$ is not 
necessarily monotonic in $\mu$ since $\e^{-M(t,t+\ell)}$ is decreasing in $\mu$. The following result provides a useful comparison test.

\begin{Prop}\label{Prop:20.9a}
For a  different inhomogeneous Poisson process $\N'$ with rate function $\mu'(t)>0$, it holds that:\\[0.5mm]
{\rm (i)} If $\mu\le\mu'$  where $I'=\infty$ or $M'(0,\infty)<\infty$, then $\Prob(D<\infty)=1$.\\
{\rm (ii)} If $\mu\ge\mu'$ where $I'<\infty$ and $M'(0,\infty)=\infty$, then $\Prob(D=\infty)>0$.
%\\[3mm]
%{\rm (i)} If $\mu\le\mu'$ and $I=\infty$, then $\Prob(D'<\infty=1)$.\\
%{\rm (ii)} If $\mu'\le\mu$ and $I<\infty$, $M'(0,\infty)=\infty$, then $\Prob(D'=\infty)>0$.
\end{Prop}
\begin{proof} This follows by standard coupling arguments.
In (i), write $\N'=\N+\N''$  where $\N,\N''$ are independent and $\N''$ Poisson with rate function $\mu'-\mu$.
Then $\N\subseteq\N'$ which immediately implies $D\le D'<\infty$. The proof of (ii) is similar. %, noting that $M'(0,\infty)<\infty$ has to be excluded since then $\Prob(D'<\infty)=1$.
\end{proof}

%\begin{Rem}\label{Rem:20.9a}\normalfont
%  Nevertheless,
%the converse parts of Theorem~\ref{prop_int_test}  give that also $I=\infty$ or $M(0,\infty)<\infty$ under the conditions of (i)
%and $I<\infty$ and $M(0,\infty)=\infty$ under the conditions of (i).
%\halmoss
%\end{Rem}

\begin{Cor}\label{cor:finite}
It holds that:\\
{\rm (i)} If $\limsup_{t\to\infty}\mu(t)/\log t<1/\ell$ then
$D<\infty$ a.s.\\
{\rm (ii)} If $\liminf_{t\to\infty}\mu(t)/\log t>1/\ell$ then $\Prob(D=\infty)>0$.
\end{Cor}
\begin{proof}
Consider for some $h>0$ the particular rate function
\[\mu'(t)=\indi{t>T}h\log t/\ell+\indi{t\leq T}\mu(t)\]
with $h$, $T$ chosen such that $\mu\le\mu'$, $h<1$ in (i) and $\mu\ge\mu'$, $h>1$ in (ii).
 Observe that $M'(t,t+\ell)=h\log t+\oh(1)$.
 Hence $\e^{-M'(t,t+\ell)}=t^{-h}(1+\oh(1))$, and so $I'$ is finite for $h>1$ and infinite for $h<1$.
 Reference to Proposition~\ref{Prop:20.9a} completes the proof.
\end{proof}

\begin{Rem}\label{rem:finite}\rm
Consider the case $\mu(t)\sim\log t/\ell$ where the results of Corollary~\ref{cor:finite} do not apply. We may still deduce from Theorem~\ref{prop_int_test} that $D<\infty$ a.s.\ whenever $\mu(t)\leq \log t/\ell+a$ for all large $t$ and some $a<\infty$.
%The case left open is essentially where $\ell\mu(t)-\log t$ goes to $\infty$ but at a slower rate than $\epsilon\log t$ for any small~$\epsilon>0$.
More generally, we have the following result. Let
\[\ell\mu(t)=\log t+2\log_2t+\log_3t+\ldots+\log_{n-1}t+b\log_n t\]
for all large $t$ and some $n\geq 4$, where $\log_n$ denotes the $n$-fold iterated logarithm.
Then $D<\infty$ a.s.\ if and only if $b\leq 1$.

To see the above, note that $M(t,t+\ell)=\ell\mu(t)+\oh(1)$ and then $I=\infty$ if and only if $\int_c^\infty e^{-\ell\mu(t)}\mu(t)\D t=\infty$. But the latter is equivalent to divergence of the integral
 \begin{align*}\int_c^\infty\frac{1}{t\log t\cdots\log_{n-2}t\log_{n-1}^bt} \,\D t =\ 
 \begin{cases}{\displaystyle \frac{1}{1-b}\Bigl[\log_{n-1}^{1-b}t\Bigr]_c^\infty}&b\neq 1\\[2mm]
 \Bigl[\log_{n} t\Bigr]_c^\infty=\infty&b=1.\end{cases}\end{align*}
 Finally, $M(0,\infty)=\infty$ and so the result follows from Theorem~\ref{prop_int_test}.
 
 Similar calculations (or comparisons) shows that $\ell\mu(t)=\log t+(1+a)\log_2t$ or $\ell\mu(t)=\log t+2\log_2t+a \log_3t$ both lead to $D<\infty$ if and only if 
 $a\le 1$.
\end{Rem}

\section{When does $D$ have a long tail?}\label{sec:long_short}
In the heavy-tailed area, it is customary to call a r.v.\ $X$ for \emph{long-tailed}
if $\Prob(X>t+u)/\Prob(X>t)\to 1$ for any $u>0$ as $t\to\infty$. This contrasts typical
light-tailed r.v.'s such as the gamma or inverse Gaussian where the limit is in $(0,1)$,
or the Gaussian or light-tailed Weibull, i.e.\ $\Prob(X>t)=\e^{-t^\beta}$ with $\beta>1$, where it is 0. In the present context, we have:

\begin{Prop}\label{prop:long_short} Let $u>0$. As $t\rightarrow\infty$ it holds that:\\
{\rm (i)} $\PP(D>t+u)/\PP(D>t)\rightarrow 1$ if $\mu(t)\rightarrow\infty$.\\
Moreover, if $\liminf M(t,t+\ell-\epsilon)>0$ for some $\epsilon>0$ then $\liminf\PP(D>t+u)/\PP(D>t)>0$. \\
{\rm (ii)} Conversely, let $u\geq\ell$. Then $\PP(D>t+u)/\PP(D>t)\rightarrow 0$ if $\mu(t)\rightarrow 0$.\\
Moreover, if $\limsup M(t,t+\ell)<\infty$, then $\limsup\PP(D>t+u)/\PP(D>t)<1$.
\end{Prop}
\begin{proof}
According to Remark~\ref{rem:l=1} we may assume that $\ell=1$. In (ii), $u\geq 1$ and
$D>t+u$ imply that at least $D>t$ and there is a point in $(t+1,t+2)$. Thus
\[\PP(D>t+u)\leq\PP(D>t)(1-\e^{-M(t+1,t+2)}),\]
where independence follows from the fact that $D>t$ is determined by the points of $\N$ in $(0,t+1]$.
From this both assertions of (ii) follow, noting that
if $\mu(t)\rightarrow 0$ then also $M(t+1,t+2)\rightarrow 0$.

For (i), we first note that for $u_1,u_2>0$ we have $u_2\in\bigl[(k-1)u_1,ku_1\bigr]$ for some $k=1,2,\ldots$
and so
\begin{align*}\MoveEqLeft \frac{\Prob(D>t+u_2)}{\Prob(D>t)}\ \ge \ \frac{\Prob(D>t+ku_1)}{\Prob(D>t)}\ =\ 
\prod_{i=1}^k\frac{\Prob(D>t+iu_1)}{\Prob(D>t+(i-1)u_1)}\,.
%\frac{\Prob(D>t+ku_1)}{\Prob(D>(k-1)u_1)}\,\frac{\Prob(D>t+(k-1)u_1)}{\Prob(D>(k-2)u_1)}\cdots\
%\frac{\Prob(D>t+u_1)}{\Prob(D>t)}\,.
\end{align*}
Thus if any of the two assertions hold for $u_1$, it holds also for $u_2$ and so we can take $u\in(0,1)$.

For a given $t$, let $\tau_t=\inf\bigl\{s\in\N:\,s\in (t,t+1], \text{no gap in }(0,s)\bigr\}$, $\tau_t=\infty$ if no such $s$ exists.
%Assume $ M(t,t+1-h)>\delta>0$ for $t\ge T$. 
The event $D>t+u$ will hold if either $\tau_t\in(t,t+u]$ and there is
a point in $(t+u,t+1]$ (the first candidate for $D$ is then the first such point), or if $\tau_t\in(t+u,t+1]$ 
 ($\tau_t$ is then the first candidate for $D$). Thus
\begin{align*}\Prob(D>t+u)\ &\ge\ \Prob\bigl(\tau_t\in(t,t+u]\bigr)(1-\e^{-\delta})+
 \Prob\bigl(\tau_t\in(t+u,t+1]\bigr)\\
 &\ge\ \Prob\bigl(\tau_t\in(t,t+1]\bigr)(1-\e^{-\delta}) =\ \Prob(D>t)(1-\e^{-\delta}),
\end{align*}
 where $\delta=M(t+u,t+1]$. Now the first assertion in (i) is obvious, and the second follows  since $u$ can be taken arbitrarily small,
 as noted above.
%The event $D>t+1$ will hold if either $\tau\in[t,t+h]$ and there is
%a point in both of $(t+h,t+1]$ and $(t+1,t+1-h]$, or if $\tau\in(t+h,t+1]$ and there is
%a point in both of $(\tau,\tau+1-h]$ and $(\tau+1-h,\tau+2-2h]$. 
%Appying the strong Markov property at $\tau$ in the second case, we get
%\begin{align*}\Prob(D>t+1)\ &\ge\ \Prob\bigl(\tau\in[t,t+h]\bigr)(1-\e^{-\delta})^2+
% \Prob\bigl(\tau\in[t+h,]\bigr)(1-\e^{-\delta})^2\\
% &=\ \Prob\bigl(\tau\in[t,t+1]\bigr)(1-\e^{-\delta})^2\ =\ \Prob(D>t)(1-\e^{-\delta})^2
%\end{align*}
%for $t\ge T$. This immediately gives the second assertion in (i), and for the second we just need to note that
%$\Prob(D>t+1)\le\Prob(D>t)$ and that $\delta$ can be taken arbitrarily large if $\mu(t)\to\infty$.
\end{proof}

\section{Asymptotics}\label{sec:asymptotics}

%We start by noting %a rather obvious-sounding statement showing 
%that the  asymptotics of the tail of $D$ do not depend on the initial shape of~$\mu(t)$. 
%
%%The log asymptotics maybe different for $\mu'\sim\mu$ in the case of Long tail.
%\begin{Prop}\label{prop:partlyCoincide}
%Let $\mu,\mu'$ be rate functions satisfying the assumptions stated in Section~\ref{S:Intr}
%and $I=\infty$.
%If $\mu'(t)=\mu(t)$ for all $t\ge T$ and some $T<\infty$, then $\PP(D'>t)=\gamma\PP(D>t)$
%for some constant $\gamma\ne 0$ and all $t>T+\ell$. 
%\end{Prop}
%\begin{proof}
%%If $I<\infty$, then also
%We use a coupling such that $\N,\N'$ evolve (say) independently on $[0,T]$ and have the same epochs on $(T,\infty)$.
%Define $\tau$ as the time of the first common epoch after $T$; then $I=\infty$ ensures $\tau<\infty$. Consider the events
%$A_0=\{\tau\le T+\ell\}$,
%\[A_0=\bigl\{ \N\text{ has no gaps in }[0,\tau]\bigr\}, \quad A'_0=\bigl\{ \N'\text{ has no gaps in }[0,\tau]\bigr\}
%\]
%$A(t)=\bigl\{ \N\text{ has no gaps in }[\tau,t+\ell]\bigr\}$. If $t>T+\ell$, then $\{D>t\}=A_0A(t)$ for $t>T+\ell$ (note that $\tau\le T+\ell$ if $D>t$)
%and so by independence properties of the Poisson process, $\PP(D>t)=\PP A_0\cdot\PP A(t)$.
%Similarly, $\PP(D'>t)=\PP A'_0\cdot\PP A(t)$. Take $\gamma=\PP A'_0/\PP A_0$.
%\end{proof}

We start by proving that (as expected), the  asymptotics of the tail of $D$ is relatively unsensitive to the initial shape of~$\mu(t)$. 
%Note that there is no reason to expect a similar statement for finer asymptotics, see also Example~\ref{example1}.
%The log asymptotics maybe different for $\mu'\sim\mu$ in the case of Long tail.
\begin{Prop}\label{prop:partlyCoincide}
Let $\mu(t)$ be such that $D<\infty$ a.s.
Assume that $\mu'(t)$ coincides with $\mu(t)$ for all $t>T$, and that it also satisfies the assumption stated in Section~\ref{S:Intr}. Then there exist
$c_->0,c_+<\infty $ such that $c_-\PP(D'>t)\le \PP(D>t)\le c_+\PP(D'>t)$
for all $t>T$. In particular, $\PP(D>t)\logsim\PP(D'>t)$.
\end{Prop}
\begin{proof}
Using Slivnyak's formula and considering the last point in $(T,T+\ell)$ we obtain for $t>T$:
\[\PP(D>t)=\int_T^{T+\ell}\PP(\text{no gaps in }(0,s))\PP(\text{no gaps in }(s,t+\ell),\mathcal N\cap(s,T+\ell)=\emptyset)\mu(s)\D s.\]
By replacing $\PP(\text{no gaps in }(0,s))$ with $\PP'(\text{no gaps in }(0,s))$, we obtain $\PP(D'>t)$. Thus
\begin{align*}\PP(D>t)&\leq \PP(D'>t)\sup_{s\in(T,T+\ell)}\frac{\PP(\text{no gaps in }(0,s))}{\PP'(\text{no gaps in }(0,s))}
\end{align*}
and we can take $c_+=\PP(\text{no gaps in }(0,T))/\PP'(\text{no gaps in }(0,T+\ell))$ which is finite since
both probabilities are positive. The proof of the lower bound is similar. For logarithmic asuymptotics,
just note that $\log\PP(D>t)\rightarrow -\infty$.
\end{proof}
Let us note that the conclusion $\PP(D>t)\logsim\PP(D'>t)$ of Proposition~\ref{prop:partlyCoincide} does not carry over to exact asymptotics.
%\begin{Ex}\label{example1}\rm The conclusion $\PP(D>t)\logsim\PP(D'>t)$ does not carry over to exact asymptotics. For an example, take
%$\ell=1$, $\mu(t)\equiv 1$ and $\mu'(t)=\indi{t\in[0,1)}/2+\indi{t\geq 1}$. That is, we decrease the rate on the interval~$[0,1)$. Then
%\begin{equation}\label{eq:ex1}\frac{\PP(D>t)}{\PP(D'>t)}\geq 2\e^{-1/2}>1\qquad \text{ for }t>1\end{equation}
%implying that $\PP(D>t)\nsim\PP(D'>t)$.
%Indeed, onsidering the first point in $[0,1)$ we get
%\begin{align*}&\PP(D'>t)=\int_0^1\e^{-s/2}\PP'(\text{no gaps in }(s,t+1))\frac{1}{2}\D s\\&\leq \int_0^1 \e^{-s}\PP(\text{no gaps in }(s,t+1))\frac{\e^{s/2}}{2}\D s\leq \PP(D>t)\frac{\e^{1/2}}{2}.
%\end{align*}
%showing~\eqref{eq:ex1}.
%\end{Ex}

\begin{Rem}\rm
In general $\mu'(t)\sim\mu(t)$ does not imply the same logarithmic asymptotics (at least when $D$ has a long tail). A simple example is obtained using Corollary~\ref{Cor:25.7a} (i) for different~$a>0$. Nevertheless, this implication is true in the case when $\mu(t)$ has a limit in $(0,\infty)$, see Section~\ref{sec:homogen}, and in the case $\mu(t)\rightarrow 0$ for all the examples considered in Section~\ref{sec:short}.
\end{Rem}

\subsection{The homogeneous case}\label{sec:homogen}
The following result appears in~\cite{A5} and its discrete time analogue can be found in e.g.~\cite[Ch.\ XIII]{feller1}, 
but we present it here for completeness. 
\begin{Prop}\label{prop:homogen} Assume $\mu(t)\equiv \mu$ and
let $\gamma>0$ denote the unique root of
\begin{equation}\label{wc.17.6a}
1\ =\ \int_0^\ell \mu\e^{(\gamma-\mu)s}\, \dd s\,.
\end{equation}
Then $\Prob(D>t)\sim c\e^{-\gamma t}$ as $t\to\infty$ for some $0<c<\infty$\,.
\end{Prop}
\begin{proof}
Conditioning on the first epoch we write
\[\PP(D>t)\ =\int_0^{t\wedge \ell} \PP(D>t-s)\,\mu\e^{-\mu s}\,\dd s+\indi{t<\ell}\PP(T_1\in(t,\ell)).\]
Putting $Z(t)=\PP(D>t)\e^{\gamma  t}$ and $z(t)=\indi{t<\ell}\PP(T_1\in(t,\ell))\e^{\gamma t}$ we obtain the renewal equation
\begin{equation}\label{11.9a}Z(t)=z(t)+\int_0^tZ(t-s)G(\dd s),\end{equation}
where $G$ is the measure with density $\mu\e^{(\gamma-\mu) s}$ for $s\le \ell$. Observe  that $G$ is nonlattice and proper according to~\eqref{wc.17.6a}, and so the key renewal theorem, see~\cite[Thm.\ V.4.3]{APQ}, shows that
$Z(t)\rightarrow\int_0^\infty z(s)\dd s/\int_0^\infty s G(\dd s)=c\in(0,\infty)$. This completes the proof.
\end{proof}

\begin{Rem}\label{rem:gamma}
Solving~\eqref{wc.17.6a} corresponds to finding a unique positive $\gamma\neq \mu$ such that $\gamma \e^{-\gamma \ell}=\mu \e^{-\mu\ell}$, unless $\mu=1/\ell$ in which case  $\gamma=\mu$. In particular, $\gamma$ is a continuous decreasing function of~$\mu$.
\end{Rem}

%Next, recall Proposition~\ref{prop:long_short} where we have considered the two cases $\mu(t)\rightarrow\infty$ and $\mu(t)\rightarrow 0$. It is natural to ask what happens if $\mu(t)\rightarrow\mu\in(0,\infty)$? The following result answers this question.
\begin{Cor}\label{cor:constlim}
If $\mu(t)\rightarrow\mu\in(0,\infty)$ then $\PP(D>t)\logsim \e^{-\gamma t}$, where $\gamma>0$ is the root of~\eqref{wc.17.6a}. Moreover, $-\log\PP(D>t)/t$ converges to 0 or $\infty$ according to $\mu(t)\rightarrow \infty$ and $\mu(t)\rightarrow 0$.
\end{Cor}
\begin{proof}
Given a small $\epsilon>0$, choose $T$ such that $\mu(t)\in(\mu-\epsilon,\mu+\epsilon)$ for all $t>T$. We may write
$\PP(D>t)\leq \PP(D'>t)$,
where $D'$ corresponds to $\mu(t)$ being fixed at $\mu+\epsilon$ for all $t>T$. According to Proposition~\ref{prop:partlyCoincide} and Proposition~\ref{prop:homogen} we have $\PP(D'>t)\logsim \e^{-\gamma' t}$, where $\gamma'$ solves~\eqref{wc.17.6a} with $\mu$ replaced by $\mu+\epsilon$. Thus
\[\liminf_{t\rightarrow \infty}\frac{\log \PP(D>t)}{-\gamma t}\geq \frac{\gamma'}{\gamma}\]
for any $\epsilon>0$. But $\gamma'\uparrow\gamma$ as $\epsilon\downarrow 0$  according to Remark~\ref{rem:gamma}, and hence the lower bound of $\liminf$ is~1. Similar reasoning shows that 1 is also the upper bound for the $\limsup$ which completes the first part of the proof.

If $\mu(t)\rightarrow 0$, we may choose some arbitrarily small $\mu>0$ and some $T$ such that $\mu(t)<\mu$ for $t>T$. Similarly as above we have
\[\liminf_{t\rightarrow\infty}-\frac{1}{t}\log\PP(D>t)\geq \gamma',\]
where $\gamma'$ solves~\eqref{wc.17.6a}, but then $\gamma'\rightarrow\infty$ as $\mu\downarrow 0$. The case of $\mu(t)\rightarrow\infty$ is similar.
\end{proof}

\subsection{Short tail}\label{sec:short}
Throughout this section we assume that $\mu(t)\rightarrow 0$ and so $\PP(D>t+u)/\PP(D>t)\rightarrow 0$ according to Proposition~\ref{prop:long_short}(ii); we call such $D$ \emph{short-tailed}. Corollary~\ref{cor:constlim} shows that if $\PP(D>t)\logsim \e^{-f(t)}$ then it must be that $t/f(t)\rightarrow 0$ as $t\rightarrow \infty$. 
This property of $f(t)$ turns out to be crucial in tightening the gap between the bounds for $\log\PP(D>t)$ proposed below. %Moreover, under this assumption it is reasonable to expect short tails, and thus 
%In fact, we mainly think about the case $\mu(t)\rightarrow 0$, which does imply the above property of $f$ according to . %Finally, Remark~\ref{rem:l=1} allows to take $\ell=1$ for notational simplicity.

%Before we proceed let us note that the above assumptions on $\mu(t)$ imply that $a\mu(t)$ leads to the same logarithmic asymptotics of the tail of~$D$ irrespective of~$a>0$, which is demonstrated by particular examples in Corollary~\ref{Cor:25.7c}. Moreover, this together with Proposition~\ref{prop:partlyCoincide} implies that $\mu'(t)\sim\mu(t)$ leads to $\PP(D'>t)\logsim\PP(D>t)$. 

\begin{Th}\label{thm:short}
Assume that $\mu(t)$ is eventually non-increasing and let $f$ be a function which is
regularly varying at $\infty$ with index $\alpha\geq 1$ and satisfies $t/f(t)\rightarrow 0$. If 
\begin{equation}\label{eq:cor_short}\sum_{i=1}^n\log(\mu(ic))\ \sim\  -c^\rho f(n) \text{ as }n\rightarrow\infty\end{equation}
for all $c>0$ and some $\rho\in\RR$, then $\PP(D>t)\ \logsim\  \e^{-\ell^{\rho-\alpha} f(t)}$ as $t \rightarrow\infty$.
\end{Th}
\noindent Note that the assumption  $t/f(t)\rightarrow 0$ is automatic if $\alpha>1$.
 The proof of this result relies on the following Lemma.
\begin{Lemma}\label{lem:short}
Assume that $\ell=1$, $M(t,t+1)$ is bounded for large $t$ and that there exist positive functions $f,g$ such that $t/f(t)\rightarrow 0$ as $t\rightarrow\infty$ and $g(\epsilon)\rightarrow 1$ as $\epsilon\downarrow 0$. If for any small enough $\epsilon>0$ and some $k$
\begin{align}\label{eq:cond1}
&\limsup_{t\rightarrow\infty}\frac{1}{f(t)}\sum_{i=k}^{\lfloor t\rfloor}\log M(i,i+1)\ \leq\ -1,\\
\label{eq:cond2}
&\liminf_{t\rightarrow\infty}\frac{1}{f(t)}\sum_{i=k}^{\lceil t/(1-\epsilon)\rceil}
\log M\bigl(i(1-\epsilon),i(1-\epsilon)+\epsilon\bigr)\ \geq\ -g(\epsilon)
\end{align}
then $\PP(D>t)\logsim \e^{-f(t)}$ as $t \rightarrow\infty$.
\end{Lemma}
\begin{proof}
Notice that $\{D>t\}$ is contained in the event that $\N$ has a point in each of the intervals $[i,i+1),i=k,\ldots,\lfloor t\rfloor$, and so
\[\Prob(D>t)\ \leq\ \prod_{i=k}^{\lfloor t\rfloor}\bigl(1-\exp\{-M(i,i+1)\}\bigr)\ \leq\ 
\prod_{i=k}^{\lfloor t\rfloor}M(i,i+1).\]
by the independence property of $\N$  and \eqref{23.9a}.
Taking logarithms we get
\[\log\PP(D>t)\ \leq\ \sum_{i=k}^{\lfloor t\rfloor}\log M(i,i+1).\]

Fix $\epsilon\in (0,1/2)$ and let $h=1-\epsilon$. Consider the intervals $[hi,hi+\epsilon)$ and assume that there is a point of $\N$ in each of these intervals for $i=0,\ldots,n$ with $n=\lceil t/h\rceil$. Then necessarily $D>t$, because any interval $[s,s+1)$ for $s\in[0,t]$ must contain one of the intervals $[hi,hi+\epsilon)$ and hence a point (see also Figure~\ref{fig:bound}).
\begin{figure}[h]
  \centering
  \includegraphics{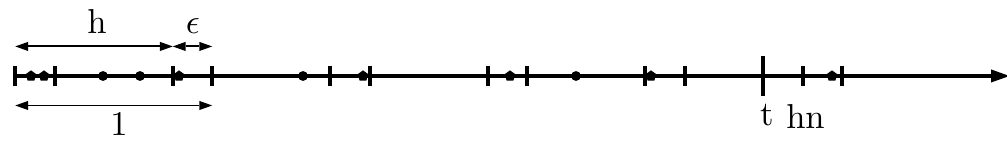}
  \caption{Points of $\N$ and the intervals $[hi,hi+\epsilon)$.}
  \label{fig:bound}
\end{figure}
Thus we have a bound
\[\Prob(D>t)\ \geq\ \prod_{i=0}^{\lceil t/h\rceil}\bigl(1-\exp\{-M(hi,hi+\epsilon)\}\bigr)\ \geq\  
c_1\prod_{i=k}^{\lceil t/h\rceil}c_2M(hi,hi+\epsilon)\]
for some $c_1,c_2>0$, because $M(t,t+1)$ is eventually bounded. Hence we obtain
\[\log\PP(D>t)\ \geq\ \Oh(t)+\sum_{i=k}^{\lceil t/h\rceil}\log M(hi,hi+\epsilon).\]

Finally, from~\eqref{eq:cond1} and~\eqref{eq:cond2} we get
\[-g(\epsilon)\ \leq\ \liminf \log\PP(D>t)/f(t)\leq\limsup \log \PP(D>t)/f(t)\leq -1,\]
because $t/f(t)=\oh(1)$ according to the assumptions. 
Choosing $\epsilon$ arbitrarily small we get $\log\PP(D>t)/f(t)\to-1$, which completes the proof.
\end{proof}

%Let us give a simplified version of Proposition~\ref{prop:short} for an important special case.

\begin{proof}[Proof of Theorem~\ref{thm:short}]
First we fix $\ell=1$ and show that $\PP(D>t)\logsim \e^{-f(t)}$.
Regular variation implies that $f(t+h)/f(t)\sim 1$ uniformly in $h\in[-1,1]$ (use e.g.\ the representation theorem~\cite[Thm.\ 1.3.1]{bingham1989regular}).
Next, we note that $M(i,i+1)\leq \mu(i)$ and $M\bigl(i(1-\epsilon),i(1-\epsilon)+\epsilon\bigr)\geq 
\epsilon\mu\bigl((i+1)(1-\epsilon)\bigr)$ for $\epsilon<1/2$ and $i\geq i_0$.
So we have
\[\frac{1}{f(t)}\sum_{i=i_0}^{\lfloor t\rfloor}\log M(i,i+1)\ \leq\
 \frac{1}{-f(\lfloor t\rfloor)}\sum_{i=i_0}^{\lfloor t\rfloor}\log(\mu(i))\frac{-f(\lfloor t\rfloor)}{f(t)}\ \sim\ -1.\]
Similarly,
\begin{align*}
&\frac{1}{f(t)}\sum_{i=i_0}^{\lceil t/(1-\epsilon)\rceil}\log M\bigl(i(1-\epsilon),i(1-\epsilon)+\epsilon\bigr)\ 
\geq\ \frac{1}{f(t)}\sum_{i=i_0+1}^{\lceil t/(1-\epsilon)\rceil+1}\bigl(\log \mu\bigl(i(1-\epsilon)\bigr)+\log \epsilon\bigr)\,.
\end{align*}
It only remains to note that
\[f\bigl({\lceil t/(1-\epsilon)\rceil}+1\bigr)/f(t)\ \sim\ f\bigl(t/(1-\epsilon)\bigr)/f(t)\ \sim\ (1-\epsilon)^{-\alpha},\]
and so we take $g(\epsilon)=(1-\epsilon)^{\rho-\alpha},$
which concludes the proof for $\ell=1$.

For arbitrary $\ell>0$,  we take $\mu'(t)=\ell\mu(t\ell)$ according to Remark~\ref{rem:l=1} and so
\[\sum_{i=1}^n\log\mu'(ic)\ =\ n\log \ell+ \sum_{i=1}^n\log\mu(ic\ell)\ \sim\ -c^\rho \ell^\rho f(n)\]
implying $\PP(D>t)=\PP(D'>t/\ell)\logsim \e^{-\ell^\rho f(t/\ell)}\logsim \e^{-\ell^{\rho-\alpha} f(t)}$.
\end{proof}
\begin{Rem} \rm
In the above proof the choice of the lower bound for $\PP(D>t)$ is not initially obvious: it requires a point near each natural number up to $\lceil t\rceil$, see Figure~\ref{fig:bound}, and so it may look too weak.
It could appear more natural  to require that each of the intervals $[i/2,i/2+1/2), i=0,\ldots, 2\lceil t\rceil$ has a point. In the short tail case it turns out that it is essential to keep the number of terms in the product close to $t$, which is achieved by the former bound.
\halmoss
 \end{Rem}

Specializing to some important forms of $\mu(t)$, we obtain:
\begin{Cor}\label{Cor:25.7c}
Let $a,b>0$. As $t \rightarrow\infty$ it holds that:\\ [1mm]
{\rm (i)} If $\mu(t)=a\log^{-b} t$ for large $t$ then $\PP(D>t)\logsim \e^{-bt\log\log t/\ell}$.\\
{\rm (ii)} If $\mu(t)=at^{-b}$ for large $t$ then $\PP(D>t)\logsim \e^{-bt\log t/\ell}$.\\
{\rm (iii)} If $\mu(t)=a\e^{-bt}$ for large $t$ then $\PP(D>t)\logsim \e^{-bt^2/(2\ell)}$.
\end{Cor}
%It is noted that in case (ii) for $b\geq 1$ we need to truncate the rate by putting $\mu(t)=at^{-b}\wedge M$ to ensure that
%$M(0,t)<\infty$. The asymptotic behaviour is not affected by this truncation.
\begin{proof} (i) Note that $\sum_{i=1}^n\log \log(ic)\sim n\log\log n$,
where one can use e.g.\ the discrete L'Hospital's rule.
Hence
\[\sum_{i=1}^n\log\mu(ic) \ =\ n\log a-b\sum_{i=1}^n\log \log(ic)\sim - bn\log\log n\ =\ -f(n).\]
Hence $\rho=0,\alpha=1$ and $\PP(D>t)\logsim \e^{-bt\log\log t/\ell}$.

(ii) Observe that
\[\sum_{i=1}^n\log\mu(ic)\ =\ n(\log a-b\log c)-b\sum_{i=1}^n\log i\ \sim\ - bn\log n\ =\ -f(n).\] Hence $\rho=0,\alpha=1$ and $\PP(D>t)\logsim \e^{-bt\log t/\ell}$.

%(i) Follows by an argument which is similar and hence omitted.

(iii) Here
\[\sum_{i=1}^n\log\mu(ic)\ =\ n\log a-bc\sum_{i=0}^n i\sim -bc n^2/2\ =\ -cf(n).\]
Hence $\rho=1,\alpha=2$ and so $\PP(D>t)\logsim\e^{-bt^2/(2\ell)}.$
\end{proof}

\subsection{Long tail}\label{SS:Long_tail}
In this section we focus on the case of long tails, and hence we mainly think about examples where $\mu(t)\rightarrow\infty$
but so slowly that $D<\infty$ a.s..
Recall from Corollary~\ref{cor:finite} that this essentially only allows $\mu(t)$
to grow at rate $\log t$. This makes it reasonable to assume that $M(t,t+\ell)$ is approximately $\ell\mu(t)$. It turns out that in the long-tailed case the delay differential equation~\eqref{eq:diff_eq} readily provides the asymptotics.
\begin{Prop}\label{prop:long}
Assume that $\mu(t)$ is continuous almost everywhere, tending to $\infty$, and  satisfies
%$\limsup\mu(t)/\log t<1,
$M(t,t+\ell)=\ell\mu(t)+\oh(1)$ as $t\rightarrow\infty$ together with
the conditions of Theorem~{\normalfont \ref{prop_int_test}} for $\Prob(D<\infty)=1$.
Let $c$ be arbitrary and define $f(t)=\int_c^t
\e^{-\ell\mu(s)}\mu(s)\,\dd s$. Then $\Prob(D>t)\logsim\e^{-f(t)}$.
%If there is a function $f$ such that $f'(t)=\e^{-\mu(t)}\mu(t)$ for all $t$ large enough then
%\[-\log\PP(D>t)\sim f(t). \]
\end{Prop}
\begin{proof}
Let $\xi(t)=-\log\PP(D>t)$.
According to~\eqref{eq:difflog} we may write
\[\xi'(t)=\e^{-M(t,t+\ell)}\frac{\PP(D>t-\ell)}{\PP(D>t)}\mu(t)\sim \e^{-\ell\mu(t)}\mu(t),\]
where we also employ Proposition~\ref{prop:long_short} giving that $\PP(D>t-\ell)/\PP(D>t)\rightarrow 1$. Hence for any $\epsilon>0$ there exists $T$ such that
\[(1-\epsilon)\e^{-\ell\mu(t)}\mu(t)\ \leq\ \xi'(t)\ \leq\ (1+\epsilon)\e^{-\ell\mu(t)}\mu(t), \quad t>T.\]
Thus by the fundamental theorem of calculus we have for $t>T$:
\[(1-\epsilon)\bigl(f(t)-f(T)\bigr)\ \leq\ \xi(t)-\xi(T)\ \leq \ (1+\epsilon)\bigl(f(t)-f(T)\bigr).\]
But $\xi(t)\to\infty$ since $\Prob(D<\infty)=1$, and so it must be that $f(t)\rightarrow\infty$.
This gives $\xi(t)\sim f(t)$.
\end{proof}
\begin{Cor} \label{Cor:25.7a}
The following logarithmic asymptotics hold:\\[1mm]
{\rm (i)} If $\mu(t)=b\log t-\log a$ for large $t$ with $b\in (0,1/\ell)$ and $a>0$ then
\[\PP(D>t)\logsim \exp\Bigl\{-\frac{ba^\ell}{1-b\ell}t^{1-b\ell}\log t\Bigr\}\,.\]
{\rm (ii)} If $\ell=1$ and $\mu(t)=\log t-b\log\log t-\log a$ for large $t$ with $a>0,b>-2$ then
\[\PP(D>t)\logsim \exp\Bigl\{-\frac{a}{2+b}\log^{2+b} t\Bigr\}\,.\]
{\rm (iii)} If $\ell=1$ and $\mu(t)=\log t+2\log\log t-\log a$ for large $t$ with $a>0$ then
\[\PP(D>t)\logsim \log^{-a} t\,.\]
\end{Cor}
\begin{proof}
In (i), the conditions of Proposition~\ref{prop:long} concerning $\mu(t)$ are easily verified. Finally, $\e^{-\ell\mu(t)}\mu(t)=a^\ell t^{-b\ell}(b\log t-\log a)$ with primitive \[f(t)\ =\ \frac{ba^\ell}{1-b\ell}t^{1-b\ell}(\log t-1/(1-b\ell))-\frac{a^\ell\log a}{1-b\ell}t^{1-b\ell}\ \sim\ \frac{ba^\ell}{1-b\ell}t^{1-b\ell}\log t\,.\]
For (ii), we get
\begin{align*}
\int_c^t \e^{-\mu(s)}\mu(s)\,\dd s\ &\sim
\int_c^t \frac{a\log^b s}{s}\log s\,\dd s\ =\ \Bigl[\frac{a\log^{2+b} s}{2+b}\Bigr]_c^t
\end{align*}
The assumption $b>-2$ ensures that this expression has limit $\infty$ as $t\to\infty$,
and therefore both that $\Prob(D<\infty)=1$ and that $\Prob(D>t)$ has the asserted logarithmic
asymptotics, see also Remark~\ref{rem:finite}. In (iii) we have instead 
\begin{align*}
\int_c^t \e^{-\mu(s)}\mu(s)\,\dd s\ &\sim
\int_c^t \frac{a}{s\log s}\,\dd s\ =\ \Bigl[a\log\log s\Bigr]_c^t
\end{align*}
yielding the result.
\end{proof}

\begin{Rem}\label{Rem:24.7a}\rm
Note that in (ii), the result for $b=-1$ corresponds to a power tail
$1/t^a$, the one for $b=0$ to a lognormal tail etc.

Comparing to the homogeneous case $\mu\equiv a$, (i) in Corollary~\ref{Cor:25.7a}
gives the asymptotic tail of $D$ when $\mu(\cdot)$ is of slightly
bigger order. The complementary result for slightly smaller order is 
(i) in Corollary~\ref{Cor:25.7c}.\halmoss
\end{Rem}

The problems corresponding to long and short tail asymptotics appear to be essentially different.
In particular, the delay differential equation~\eqref{eq:diff_eq}
seems to be of no help in the short tail case, whereas bounding ideas do not work in the long tail case. 
E.g.\ for $\mu(t)=b\log t$,  $\ell=1$ they only give
\[C_1t^{1-b}\leq -\log\PP(D>t)\leq C_2 t^{1-b/2},\]
while the correct answer is $Ct^{1-b}\log t$ according to~(i) in Corollary~\ref{Cor:25.7a}.

Finally, we consider finiteness of moments of~$D$.
\begin{Cor} Assume that $\ell=1$ and let 
\[\mu(t)=\log t +\log\log t-\log a, \qquad a>0\]
for large $t$.
Then $\EE D^p$ is finite if $p<a$ and infinite if $p>a$.

Moreover, if $\mu'(t)\leq \mu(t)$ for large $t$ and $p<a$ then $\EE {D'}^p<\infty$. If $\mu'(t)\geq \mu(t)$ for large $t$ and $p>a$ then $\EE{D'}^p=\infty$.
\end{Cor}
\begin{proof}
From Corollary~\ref{Cor:25.7a}(ii) with $b=-1$ it follows that $\PP(D>t)<t^{-a(1-\epsilon)}$ for large~$t$, where $\epsilon>0$ is some arbitrarily small number. Hence
\[\EE D^p=p\int_0^\infty\PP(D>t)t^{p-1}\D t< C\int_T^\infty t^{p-1-a(1-\epsilon)}\D t\]
which is finite if $p-a(1-\epsilon)<0$. Therefore, $p<a$ implies $\EE D^p<\infty$. Proof of the converse statement follows the same lines.

For the second statement define $\mu''(t)$ which coincides with $\mu(t)$ for large $t$ and otherwise with $\mu'(t)$. Consider $\EE {D''}^p$ and use a comparison argument similar to that in Proposition~\ref{Prop:20.9a} to complete the proof.
\end{proof}
As a consequence of the above result we note that $\mu(t)=\log t+b\log\log t$  for $b<1$ ensures that all the moments of $D$ are finite, and for $b>1$ that all the moments of $D$ are infinite (while $D<\infty$ a.s. for $b\leq 2$).

\section{Translation to \restart{} problems}\label{sec:Restart}

Tasks such as the execution of a computer program or the transfer
of a file on a communications link may fail. There is a
considerable literature on protocols for handling such failures.
We mention in particular \textsc{resume}\ where the task is resumed after
repair, \textsc{replace} where the task is abandoned and a new one taken
from the pile of waiting tasks, \RESTART\ where the task needs to be
restarted from scratch, and \textsc{checkpointing} where the task contains
checkpoints such that performed work is saved at checkpoint times
and that upon a failure, the task only needs to be restarted from
the last checkpoint.

The model of Asmussen {\em et al.} \cite{A5} assumes that failures
occur at a time after each restart with the same distribution $G$
for each restart (a particular important case is of course the
exponential distribution). However, it is easy to imagine situations
where the model behaviour is rather determined by the time of the
day (the clock on the wall) rather than the time elapsed since the
last restart. Think, e.g., of a time-varying load in the system
which may influence the failure rate and/or the speed at which the
task is performed. For example, the load could be identified with
the number of busy tellers in a call
center or the number of users in a LAN (local area network)
currently using the central
server, both exhibiting  rush-hours. We provide here
some first insight in the behaviour of such models.

The emphasis in \cite{A5} is on the more difficult case of a random
rather than a constant ideal task time. However, as a first attempt it seems reasonable
to assume a  constant ideal task time
of length $\ell$.
Then total task time is simply $X=\ell+D$, and the results of Sections~\ref{S: Dfin}--\ref{sec:asymptotics}
immediately apply to show that the critical
rate of increase of $\mu(t)$ for $X$ to be finite is $\log t/\ell$, that $\Prob(X>x)\logsim \e^{-bx\log x/\ell}$ as $x\rightarrow\infty$
when $\mu(t)=a/t^b$ etc.

%\subsection{Time-varying processing rate}
A model of equal interest is the one with a time-varying processing rate $r(t)\ge 0$. For convenience, we will assume that
$r(t)$ is a continuous, strictly positive function satisfying $\int_0^\infty r(t)\D t=\infty$, and that failures occur according to a Poisson process
with constant rate $\mu^*$. The quantity of interest is again the delay $D^*$
(sum of times of unsuccessful attempts).

\newcommand{\Rinv}{R^{-1}}
Note that $R(t)=\int_0^t r(s)\,\dd s$ is the amount of work that has been spent on the task
up to time $t$ provided the task has not been completed and the
task time $X^*$ in absence of failures is given by $R(X^*)=\ell$, i.e.\
$X^*=\Rinv(\ell)$. More generally, if the task is not completed at the
time $T^*_{n-1}$ of the $(n-1)$th failure, then the task is still
uncompleted  at  $T^*_n$   if and only if
$R(T^*_n)-R(T^*_{n-1})<\ell$.
Hence the results for this model follow in a straightforward way from the preceding analysis by
using the time change $T_n=R(T^*_n)$ to transform $\cal N^*$
into an inhomogeneous Poisson process $\cal N$. Then $D^*=\Rinv(D)$ and $X^*=\Rinv(D+\ell)$, and 
$M(0,R(t))=M^*(0,t)=\mu^*t$ implying $\mu(R(t))=\mu^*/r(t)$.
We do not spell out this translation for all cases, but for example:

\begin{Cor}\label{Cor:11.7a} %Assume that $\cal N$
{\rm (i)} If $\limsup_{t\to\infty}\mu^*/(r(t)\log R(t))<1/\ell$, then
$X^*<\infty$ a.s. Conversely, $\Prob(X^*=\infty)>0$ provided $\liminf_{t\to\infty}\mu^*/(r(t)\log R(t))>1/\ell$.\\
{\rm (ii)} Assume $r(t)=at^b$ with $b>0$. Then
\[\Prob(X^*>t)\ \logsim\ \Prob(D^*>t)\ \logsim\ \exp\Bigl\{-\frac{ab}{b+1}t^{b+1}\log t /\ell\Bigr\}\,.\]
\end{Cor}
\begin{proof}
According to Corollary~\ref{cor:finite} we need to look at the limiting behaviour of $\mu(t)/\log t$, which is the same as that of $\mu(R(t))/\log R(t)=\mu^*/(r(t)\log R(t))$. But $X^*<\infty$ iff $D<\infty$ which concludes the proof of (i). 

For (ii), note that $R(t)=a/(b+1)t^{b+1}$ and $\Rinv(t)=((b+1)/a t)^{1/(b+1)}$ yielding 
\[\mu(t)=\frac{\mu^*}{r(\Rinv(t))}=\frac{\mu^*}{a^{1/(b+1)}((b+1)t)^{b/(b+1)}}=ct^{-b/(b+1)}.\]
According to Corollary~\ref{Cor:25.7c} we have $\PP(D>t)\logsim \e^{-b/(b+1)t\log t/\ell}$.
Hence
\[\PP(D^*>t)=\PP(D>R(t))\logsim \exp\Bigl\{-\frac{b}{b+1}at^{b+1}\log t/\ell\Bigr\}.\]
Finally, $\PP(X^*>t)=\PP(D+\ell>R(t))$ showing that $\PP(X^*>t)$ has the same logarithmic asymptotic.
\end{proof}

%subsection{Random task time}
\begin{Rem}\rm
Motivated by problems such as transmitting files of randomly fluctuating sizes, a main problem in  \cite{A5,Olebook} is to
replace $\ell$ by a random $L$ independently of $\mathcal N$. It is assumed that $L$ is drawn once and fixed forever which corresponds to the \restart{} problem. Therefore,
\[\PP(D>t)=\int_0^\infty\PP(D(\ell)>t)\PP(L\in\D \ell),\]
where $D(\ell)$ corresponds to a deterministic gap of size~$\ell$.
Mathematically, this leads to a different and difficult set of problems except when $L$ is bounded a.s. 
Then if $\ell^*=\,$ess.sup.\,$L<\infty$, the inequality
\[\Prob(L>\ell^*-\epsilon)\PP(D(\ell^*-\epsilon)>t)\leq \PP(D>t)\leq \PP(D(\ell^*)>t)\]
allows many
of our results  to be easily generalized to the case of a random bounded $L$. The case $\ell^*=\infty$ is left for future studies.
\end{Rem}

\section{Discrete time version}\label{sec:discrete}
The following can be considered as an analogue of our gap problem in discrete time. Consider a (non-stationary) sequence $\xi_i,$ $i=1,2,\ldots$, of independent Bernoulli variables with $\PP(\xi_i=1)=p_i$, and let
for some fixed integer $\ell\ge 1$
\[D\ =\ \min\{n:\,\xi_n=\cdots=\xi_{n+\ell-1}=1\}\]
be the time of the first run of $\ell$ ones. Write $q_i=1-p_i$ and assume $0<p_i<1$ for all $i$.
Most results and proofs for this model is a straightforward translation from the inhomogeneous Poisson case
so we give only some selected analogues.

\begin{Cor} \label{Cor:11.9a}{\rm (i)}
$\displaystyle \PP(D=n+1)\ =\ q_n\PP(D>n-\ell)\prod_{j=n+1}^{n+\ell}p_j$\,.\\
{\rm (ii)} 
If $\displaystyle \sum_{i=1}^\infty \!q_i\!<\!\infty$ or $ \displaystyle I=\! \sum_{i=1}^\infty q_i\!\prod_{j=i+1}^{i+\ell}\!p_j<\infty$ 
then $D\!<\!\infty$ a.s.\  Otherwise $\PP(D=\infty)>0$.
\\
{\rm (iii)}
If $\displaystyle \limsup_{i\to\infty} \frac{-\log p_i}{\log i}<\frac{1}{\ell}$ then $I=\infty$ and if
$\displaystyle \liminf_{i\to\infty} \frac{-\log p_i}{\log i}>\frac{1}{\ell}$ then $I<\infty$.\\[1mm]
{\rm (iv)} $p_i\rightarrow 1$ implies short tail,
i.e.\ $\PP(D>i+n)/\PP(D>i)\rightarrow 0$, and $p_i\rightarrow 0$ implies long tail,
i.e.\ $\PP(D>i+n)/\PP(D>i)\rightarrow 1$ as $i\rightarrow\infty$,  where $n\ge\ell$.\\[1mm]
{\rm (v)} 
 In the homogeneous case $p_i\equiv p\in(0,1)$, $\PP(D>n)\sim c_1z^{-n}$, where $z$ solves $(1-p)z\sum_{i=0}^{\ell-1}(pz)^i=1$.
\\[1mm] {\rm (vi)} Short tail: if $p_i=\exp(-ai^{-b})$ then 
$\PP(D>n)\logsim \exp(-bn\log n/\ell)$.
\\[1mm] {\rm (vii)} Long tail: If $p_i=i^{-b},b\in(0,1/\ell)$ then  $\PP(D>n)\logsim \exp\{-n^{1-b\ell}/(1-b\ell)\}$ \,.
\end{Cor}
\begin{proof} The arguments are basically an easy adaptation of the ones for the inhomogeneous Poisson case to discrete time setting. In fact, some steps are even simpler and in particular, Slivnyak's formula is replaced by 
elementary conditioning arguments. Thus we only provide some crucial steps.

(i): Define
%${\cal F}_n\, =\,\sigma(\xi_1,\ldots,\xi_n)$ and let 
$B_n=\{\xi_{n-1}=0,\xi_n=\cdots=\xi_{n+\ell-1}=1\}$ as the event
that a run starts at $n$. 
Observe that $D=n+1$ if and only if $B_{n+1}$ occurs and there is no sequence of $\ell$ ones in $1,\ldots,n-1$. The latter event is independent of $\xi_n,\xi_{n+1},\ldots$ and hence of $B_{n+1}$. Moreover, it coincides with $D>n-\ell$, which concludes the proof of~(i).
%Then $D>n-\ell$ implies that there is at least one zero at times $n-\ell,\ldots,n-1$. A zero at $n$ then excludes
%$D\in\{n-\ell,\ldots,n-1\}$ so that $D=n+1$ if $B_{n+1}$ occurs. 
%The latter event equals $\{D>n-\ell\}$ which is ${\cal F}_{n-1}$-mesurable and there are no runs in $\{1,\ldots,n-1\}$
%and hence independent of $\xi_n,\xi_{n+1},\ldots$, we get
%$\Prob(D=n+1)=$ $\Prob(D>n-\ell)\Prob (B_n)$
%which is the same as asserted in (i).

(ii): Let $E_n=\cup_{\ell n+1}^{\ell (n+1)}B_i$ be the event that there is a zero in $\ell n,\ldots,\ell (n+1)-1$ followed by $\ell$ ones. Since such a zero is unique, we have $\PP(E_n)=\sum_{\ell n+1}^{\ell (n+1)}\Prob (B_i)$. Noting that
$I=\sum_1^\infty\Prob (B_i)$ and that $E_n,E_m$ are independent  if $|m-n|>1$, arguments similar as for Theorem~\ref{prop_int_test} give the first part. For the second, note that $I$ is also the total expected number of runs starting at $i>1$, and
use arguments similar to the ones based on \eqref{26.8a}.

%
%and so we take
%\[I\ =\ \sum_{i=1}^\infty (1-p_i)\prod_{j=i+1}^{i+\ell}p_j.\]
%If $\sum_{i=1}^\infty (1-p_i)<\infty$ or $I=\infty$ then $D<\infty$ a.s. and otherwise $\PP(D=\infty)>0$.
%as for Lemma~\ref{prop_int_test}, noting that $I=\sum_1^\infty \indi B_n$ is the expected number of terms

(iii): Considering the special $p_i=i^{-h/\ell}$ leads to $I=\infty$ for $h<1$ and $I<\infty$ for $h>1$.
The result then follows by a similar coupling argument as in the proof of Proposition~\ref{Prop:20.9a}.

(iv): The short tail part is just as  for Proposition~\ref{prop:long_short}. The long tail one is even simpler
since $\Prob(D>i+n)\ge$ $\Prob(D>i)q_{i+\ell}\cdots q_{i+n}$.

(v): Shown in~\cite[Ch.\ XIII.7]{feller1} by noting that the probability generating function 
of $D$ is rational and performing fractional expansions. Alternatively, the renewal equation approach of
Section~\ref{sec:homogen} applies; not surprisingly, the equation determining $z$ in~\cite[Ch.\ XIII.7]{feller1} is simply the discrete version of \eqref{wc.17.6a}. % determining the asymptotics  of the renewal equation \eqref{11.9a}.

(vi): Use a bounding argument similar to that in Lemma~\ref{lem:short}, and note that the lower bound on $\PP(D>n)$ is now obtained by placing zeros at each $i\ell$ where $i=1,\ldots\lceil n/\ell\rceil$.

(vii): Here~(iv) implies $\PP(D>n-\ell)\sim\PP(D>n)$ and so by (i)
\[\frac{\PP(D=n+1)}{\PP(D>n)}\ \sim\ \prod_{j=n+1}^{n+\ell}\frac{1}{j^b}\sim \frac{1}{n^{b\ell}}\,.\] 
Therefore for $n>N$
\begin{align*}\frac{\PP(D>n)}{\PP(D>N)}\ &=\ \prod_{k=N}^{n-1}\frac{\PP(D>k+1)}{\PP(D>k)}\ 
=\ \prod_{k=N}^{n-1}\Bigl(1-\frac{\PP(D=k+1)}{\PP(D>k)}\Bigr)\\
&=\ \exp\Bigl\{-(1+r_{n,N}')\sum_{k=N}^{n-1}\frac{1}{k^{b\ell}}\Bigr\}\ =\ 
\exp\Bigl\{-(1+r_{n,N}')\frac{n^{1-b\ell}-N^{1-b\ell}}{1-b\ell}\Bigr\},
\end{align*}
where $r_{n,N}'\to 0$ as $N\to\infty$ uniformly in $n>N$. From this the result follows.
\end{proof}

\begin{Rem}\label{Rem:28.8a}\rm
%To link the continuous and discrete setups one may think either of the correspondance
%$q_i\leftrightarrow \mu(i)$ (the rate of separators  of gaps/runs) or of $p_i\leftrightarrow \e^{-\mu(i)}$.
%If $\mu(t)\to 0$, these coincide asymptotically, which may explain why the only substantial difference appears in the 
%long tail asymptotics (compare (i) in Corollary~\ref{Cor:25.7a} to (vii) in Corollary~\ref{Cor:11.9a}).\halmoss
To link the continuous and discrete setups one may think of the correspondence $p_i\leftrightarrow \e^{-\mu(i)}$, where the latter is roughly the probability of no failures in $(i-1,i]$. That is, we partition the real line and group failures together.
Interestingly, the only substantial difference in the results appears in the 
long tail asymptotics (compare (i) in Corollary~\ref{Cor:25.7a} to (vii) in Corollary~\ref{Cor:11.9a}). In this regard note that a run of $\ell$ ones in the discretized framework implies a gap of size $\ell$, but the opposite is not always true.
This discrepancy becomes more pronounced when the rate of failures increases. This intuitively explains the fact that $\PP(D>n)$ decays faster in the continuous setup.

If $\mu_i\to 0$, the correspondence $p_i\leftrightarrow \e^{-\mu(i)}$ is equivalent to
$q_i\leftrightarrow \mu(i) $ where $q_i,\mu(i)$ can be interpreted as the rate of separators
of runs (gaps).
\halmoss
\end{Rem}

\end{document}